\documentclass[12pt,reqno]{amsart}

\usepackage{amssymb}

\numberwithin{equation}{section}

\newtheorem{theorem}{Theorem}[section]
\newtheorem{proposition}[theorem]{Proposition}
\newtheorem{lemma}[theorem]{Lemma}
\newtheorem{corollary}[theorem]{Corollary}

\theoremstyle{definition}
\newtheorem{definition}[theorem]{Definition}
\newtheorem{remark}[theorem]{Remark}

\begin{document}

\baselineskip=15pt

\title[Torsors on moduli spaces of parabolic bundles]{On some natural torsors over moduli 
spaces of parabolic bundles}

\author[I. Biswas]{Indranil Biswas}

\address{School of Mathematics, Tata Institute of Fundamental
Research, Homi Bhabha Road, Mumbai 400005, India}

\email{indranil@math.tifr.res.in}



\date{}

\begin{abstract}
A moduli space ${\mathcal N}$ of stable parabolic vector bundles, of rank $r$
and parabolic degree zero, on a $n$-pointed curve has two naturally occurring
holomorphic $T^*{\mathcal N}$--torsors over it. One of them is given by the moduli space
of pairs of the form $(E_*,\, D)$, where $E_*\, \in\, {\mathcal N}$ and $D$ is a
connection on $E_*$. This $T^*{\mathcal N}$--torsor has a $C^\infty$ section that
sends any $E_*\, \in\, {\mathcal N}$ to the unique connection on it with unitary
monodromy. The other $T^*{\mathcal N}$--torsor is given by the sheaf of
holomorphic connections on a theta line bundle over ${\mathcal N}$.
This $T^*{\mathcal N}$--torsor also has a $C^\infty$ section given by the Hermitian
structure on the theta bundle constructed by Quillen. We prove that
these two $T^*{\mathcal N}$--torsors are isomorphic by a canonical holomorphic map.
This holomorphic isomorphism interchanges the above two $C^\infty$ sections.
\end{abstract}

\maketitle

\tableofcontents

\section{Introduction}

Let $X$ be a compact connected Riemann surface and $S\, \subset\, X$ a finite subset.
Fix parabolic data, for a rank $r$ vector bundle, for points of $S$ such that
the total parabolic weight for entire $S$ is an integer. Let ${\mathcal N}_P(r)$ denote the
moduli space of stable parabolic bundles of rank $r$ and parabolic degree zero with
parabolic structure of the given type over the points of $S$.

Given any stable parabolic bundle $E_*\, \in\, {\mathcal N}_P(r)$ there is a unique parabolic 
connection on $E_*$ such that its monodromy homomorphism $\pi_1(X\setminus S,\, x_0)\, 
\longrightarrow\, \text{GL}(r, {\mathbb C})$ has its image contained in $\text{U}(r)$ 
\cite{MS}. The construction of this unitary connection is not algebro-geometric/complex-analytic.
To see this consider the moduli space ${\mathcal NC}_P(r)$ of pairs of the form $(E_*,\, D_E)$,
where $E_*\, \in\, {\mathcal N}_P(r)$ and $D_E$ is a parabolic connection on $E_*$. Assigning
to every $E_*\, \in\, {\mathcal N}_P(r)$ the unique parabolic connection $D_E$ on $E_*$
with unitary monodromy, we obtain a section
$$
\tau_{UP}\,\,:\,\, {\mathcal N}_P(r)\, \,\longrightarrow\,\, {\mathcal NC}_P(r)
$$
of the projection ${\mathcal NC}_P(r) \,\longrightarrow\,{\mathcal N}_P(r)$ defined by $(E_*,
\, D_E)\, \longrightarrow\, E_*$; using this projection, ${\mathcal NC}_P(r)$ is
a holomorphic torsor over ${\mathcal N}_P(r)$ for $T^*{\mathcal N}_P(r)$.
This section $\tau_{UP}$ is not holomorphic, which shows
that it can't be constructed just using the algebro-geometric or complex-analytic methods.

Take a theta line bundle $\Theta$ on ${\mathcal N}_P(r)$. Let
$$
{\mathcal U}_P\, \longrightarrow\, {\mathcal N}_P(r)
$$
be the algebraic fiber bundle defined by the sheaf of holomorphic connections on
${\mathcal N}_P(r)$; the space of holomorphic sections of ${\mathcal U}_P$ over any
open subset $U\, \subset\, {\mathcal N}_P(r)$ is the space of all holomorphic connections
on $\Theta\big\vert_U$. This ${\mathcal U}_P$ is a holomorphic torsor over ${\mathcal N}_P(r)$
for $T^*{\mathcal N}_P(r)$.

A construction of Quillen gives a Hermitian structure on the line bundle $\Theta$. The
corresponding Chern connection on $\Theta$ defines a $C^\infty$ section
$$
\tau_{QP}\,\,:\,\, {\mathcal N}_P(r)\, \,\longrightarrow\,\, {\mathcal U}_P
$$
of the above projection ${\mathcal U}_P\, \longrightarrow\, {\mathcal N}_P(r)$.
We note that the map $\tau_{QP}$ is \textit{not} holomorphic, so
it can't be constructed just using the algebro-geometric or complex-analytic methods.

While the two sections $\tau_{UP}$ and $\tau_{QP}$ can't be constructed just using the 
algebro-geometric or complex-analytic methods, there is a natural holomorphic isomorphism
between the two $T^*{\mathcal N}_P(r)$--torsors ${\mathcal NC}_P(r)$ and ${\mathcal U}_P$
that takes one section to the other. More precisely, we prove the following (see Theorem
\ref{thm2} and Corollary \ref{cor3}):

\begin{theorem}\label{thm-i}
There is a natural holomorphic isomorphism
$$\Psi_P\, :\, {\mathcal NC}_P(r)\, \longrightarrow\, {\mathcal U}_P$$
between the two holomorphic $T^*{\mathcal N}_P(r)$--torsors ${\mathcal NC}_P(r)$ and ${\mathcal
U}_P$.

The isomorphism $\Psi_P$ takes the section $\tau_{UP}$ to
the section $\tau_{QP}$, in other words,
$$\Psi_P\circ\tau_{UP}\,=\, \tau_{QP}\, .$$
\end{theorem}

The isomorphism $\Psi_P$ in Theorem \ref{thm-i} is in fact algebraic. To describe the strategy 
of its construction, it is first observed that both the $T^*{\mathcal N}_P(r)$--torsors 
${\mathcal NC}_P(r)$ and ${\mathcal U}_P$ admit algebraic section over a certain nonempty 
Zariski open subset $U_0$ of ${\mathcal N}_P(r)$. This isomorphism between ${\mathcal NC}_P(r)$ 
and ${\mathcal U}_P$ over $U_0$, given by the trivializations, actually extend 
to an isomorphism between ${\mathcal NC}_P(r)$ and ${\mathcal U}_P$ over entire ${\mathcal 
N}_P(r)$, thus producing the isomorphism $\Psi_P$.

This is a parabolic analog of earlier works \cite{BB} and \cite{BH}.

\section{Coverings with specified ramifications}

Let $Y$ be a compact connected oriented $C^\infty$ surface of genus $g$. Fix finitely many points
\begin{equation}\label{e1}
S\,=\, \{p_1,\, \cdots, \, p_n\}\, \subset\, Y\, .
\end{equation}
Fix an integer $N\, \geq\, 2$.

\begin{lemma}\label{lem1}
There is a ramified covering
$$
\varphi\, :\, Z\, \longrightarrow\, Y
$$
of degree $N$, where $Z$ is a compact connected oriented $C^\infty$ surface, such that following conditions
hold:
\begin{enumerate}
\item $\varphi$ is orientation preserving,

\item the map $\varphi$ is totally ramified over every point $p_i\, \in\, S$ in \eqref{e1}, and

\item $\text{Aut}(Z/Y)$ (the group of holomorphic automorphisms of $Z$
over the identity map of $Y$) is isomorphic to ${\mathbb Z}/N{\mathbb Z}$ and
$\text{Aut}(Z/Y)$ acts transitively on the fibers of $\varphi$ (meaning $\varphi$ is
ramified Galois with Galois group ${\mathbb Z}/N{\mathbb Z}$).
\end{enumerate}
\end{lemma}

\begin{proof}
Note that $\varphi$ is allowed to have ramifications over points of $Y\setminus S$. So 
adding some more points to $S$ in \eqref{e1} we may assume that the integer $n-1$ (see 
\eqref{e1}) is odd and coprime to $N$.

First assume that $g\,=\, 0$. So $(Y,\, S)$ is ${\mathbb C}{\mathbb P}^1$ with $n$ marked
points $\{c_1,\, \cdots,\, c_n\}$. Consider ${\mathbb C}{\mathbb P}^1$ as
${\mathbb C}\bigcup\{\infty\}$, and set $c_n\,=\, \infty$; so $c_i\, \in\, \mathbb C$
for all $1\, \leq\, i\, \leq\, n-1$. Now let $\mathcal Z$ be the
normalization of the complex projective curve
in ${\mathbb C}{\mathbb P}^1\times {\mathbb C}{\mathbb P}^1$ defined by the equation
$$
y^N\,=\, \prod_{i=1}^{n-1} (z-c_i)\, ,\ \ z\, \in\, {\mathbb C}\cup\{\infty\}\,=\, {\mathbb C}{\mathbb P}^1\, .
$$
We note that $\mathcal Z$ is a smooth complex projective curve in ${\mathbb C}{\mathbb P}^1\times {\mathbb C}{\mathbb P}^1$
of genus $\frac{1}{2}n(N-1) -N +1$; the genus is computed by
the Riemann--Hurwitz formula. The natural projection
\begin{equation}\label{e2}
\varphi_1\, :\, {\mathcal Z}\, \longrightarrow\, {\mathbb C}{\mathbb P}^1\, ,\ \ (y,\, z)\, \longmapsto\, z
\end{equation}
is of degree $N$, and it is totally ramified over $\{c_1,\, \cdots,\, c_{n-2},\, c_{n-1},\, c_n\,=\, \infty\}$.
The group consisting of all $N$-th roots of $1$ acts on ${\mathcal Z}$; the action of any
$\lambda$ with $\lambda^N\,=\, 1$ sends any
$(y,\, z)\,\in\, {\mathcal Z}$ to $(\lambda\cdot y,\, z)$. Consequently,
the map $\varphi_1$ in \eqref{e2} is ramified Galois with Galois group ${\mathbb Z}/N{\mathbb Z}$.
Also, $\varphi_1$ is orientation preserving as it is
holomorphic. This proves the lemma when $g\,=\, 0$.

For a general $g$, take an embedded disk ${\mathbb D}\, \subset\, {\mathbb C}{\mathbb P}^1$ such that
$$
{\mathbb D}\bigcap \{c_1,\, \cdots,\, c_{n-2},\, c_{n-1},\, \infty\}\,=\, \emptyset\, .
$$
So for the map $\varphi_1$ in \eqref{e2}, the inverse image $\varphi^{-1}_1({\mathbb D})$ is a disjoint
union of $N$ copies of the disk ${\mathbb D}$.

Let $Y'$ be a compact connected oriented $C^\infty$ surface of genus $g$, and let
\begin{equation}\label{ae}
Y\,=\, Y'\# {\mathbb C}{\mathbb P}^1
\end{equation}
be the connected sum constructed using the disk ${\mathbb D}$ in ${\mathbb C}{\mathbb P}^1$.
Let
$$
Z\,=\,\overbrace{Y'\# \cdots \# Y'}^{N\text{--copies}}\# {\mathcal Z}
$$
be the connected sum of ${\mathcal Z}$ with $N$--copies of $Y'$
constructed using the $N$ disks $\varphi^{-1}_1({\mathbb D})$. Use each copy of ${\mathbb D}$ in
$\varphi^{-1}_1({\mathbb D})$ to
attach a copy of $Y'$ to $\mathcal Z$; so each attachment is like the one in \eqref{ae}.

The above constructed $Z$ is a $C^\infty$ compact connected oriented surface. It's genus is
$N(g-1)+\frac{1}{2}n(N-1)+1$. The map $\varphi_1$ and the identity map of $Y'$ together produce a map
$$
\varphi\, :\, Z\, \longrightarrow\, Y
$$
of degree $N$. This map $\varphi$ is evidently ramified Galois with Galois group ${\mathbb 
Z}/N{\mathbb Z}$, and it is orientation preserving. Also, $\varphi$ is totally ramified 
over $\{c_1,\, \cdots,\, c_{n-2},\, c_{n-1},\, \infty\}$. This completes the proof.

\textbf{An alternative construction.}\, Increasing $n$ if necessary, assume that $n$ is a multiple of 
$N$, say $n\,=\, dN$. Let $Y$ be a compact Riemann surface of genus $g$, and fix a holomorphic line bundle
$L$ on $Y$ of degree $d$ such that
$$
L^{\otimes N}\,=\, {\mathcal O}_Y(p_1+\ldots + p_n)\, .
$$
Let $s\, \in\, H^0(Y, \, {\mathcal O}_Y(p_1+\ldots + p_n))$ be the holomorphic section given by the
constant function $1$ on $Y$. Now define
$$
Z\,=\, \{v\, \in\, L\,\, \mid\, \, v^{\otimes N}\, \in\, s(Y)\}\, \subset\, L\, .
$$
Let
\begin{equation}\label{v2}
\varphi\, :\, Z\, \longrightarrow\, Y
\end{equation}
be the restriction of the natural projection $L\, \longrightarrow\, Y$ to $Z\, \subset\, L$;
it is totally ramified over $\{p_1,\, \cdots, \, p_n\}$. Note that for the multiplication of
elements of $L$ by $\mathbb C$, the multiplication action of the $N$-th roots of $1$ preserves
$Z\,\subset\, L$. This multiplication action of the $N$-th roots of $1$
gives the group of deck transformations for the ramified covering map
$\varphi$ in \eqref{v2}. So $(Z,\, \varphi)$ satisfies all the conditions in the lemma.
\end{proof}

Lemma \ref{lem1} has the following immediate corollary:

\begin{corollary}\label{cor1}
Let $Y$ be a compact connected Riemann surface, and let $$S\,=\, \{p_1,\, \cdots, \, p_n\}\, \subset\, Y$$
be a finite subset. For any given integer $N\, \geq\, 2$, there is a nonconstant holomorphic map
$$
\varphi\, :\, Z\, \longrightarrow\, Y
$$
from a compact connected Riemann surface $Z$ such that
\begin{enumerate}
\item the degree of $\varphi$ is $N$,

\item $\varphi$ is a ramified Galois covering map with Galois group ${\mathbb Z}/N{\mathbb Z}$, and

\item $\varphi$ is totally ramified over every point of $S$.
\end{enumerate}
\end{corollary}

\begin{proof}
In Lemma \ref{lem1}, put a complex structure on the surface $Z$ such that the map $\varphi$ is holomorphic.
\end{proof}

Let $Z$ be a compact connected Riemann surface of genus $g$. The holomorphic cotangent
bundle of $Z$ will be denoted by $K_Z$. A \textit{theta characteristic} on $Z$ is a holomorphic
line bundle ${\mathcal L}$ on $Z$ of degree $g-1$ such that there is a holomorphic isomorphism
of line bundles
$$
{\mathcal L}^{\otimes 2}\, \stackrel{\sim}{\longrightarrow}\, K_Z\, .
$$
If $f\, :\, Z\, \longrightarrow\, Z$ is a holomorphic automorphism, and ${\mathcal L}\, \longrightarrow\, Z$
is a holomorphic line bundle such that ${\mathcal L}^{\otimes 2}$ is holomorphically isomorphic to
$K_Z$, then
$$
(f^*{\mathcal L})^{\otimes 2}\,=\, f^*({\mathcal L}^{\otimes 2})\,=\, f^*K_Z\,
\stackrel{\sim}{\longrightarrow}\, K_Z\, ;
$$
the last isomorphism is the dual $(df)^*\, :\, f^*K_Z\,\longrightarrow\, K_Z$
of the differential $df \, :\, TZ\, \longrightarrow\, f^*TZ$ of $f$.

Let $\text{Aut}(Z)$ denote the group of all holomorphic automorphisms of $Z$. Let $\Gamma\, \subset\, 
\text{Aut}(Z)$ be a finite subgroup. A theta characteristic ${\mathcal L}$ on $Z$ is 
called $\Gamma$--\textit{invariant} if the following condition holds: For every $\gamma\, \in\, \Gamma$,
the pulled back line bundle $\gamma^*{\mathcal L}$ is holomorphically isomorphic to $\mathcal L$.

\begin{lemma}\label{lem2}
Let $\Gamma\, \subset\, {\rm Aut}(Z)$ be a finite cyclic group.
Then there is a $\Gamma$--invariant theta characteristic on $Z$.
\end{lemma}

\begin{proof}
Fix a generator $\gamma\, \in\, 
\Gamma$ of the cyclic group. From \cite[p.~61, Proposition 5.2]{At2} we know that there is 
a holomorphic line bundle ${\mathcal L}$ on $Z$ such that
\begin{itemize}
\item the holomorphic line bundle ${\mathcal L}^{\otimes 2}$ is holomorphically isomorphic to
$K_Z$, and

\item $\gamma^*{\mathcal L}$ is holomorphically isomorphic to ${\mathcal L}$.
\end{itemize}
This $\mathcal L$ is a $\Gamma$--invariant theta characteristic on $Z$.
\end{proof}

\section{Parabolic bundles and equivariant bundles}

\subsection{Parabolic bundles}\label{se3.1}

We briefly recall the basic definitions regarding parabolic bundles; more details can
be found in \cite{MS}, \cite{MY}.

Let $X$ be a compact connected Riemann surface. Fix a nonempty finite subset
$$S\,=\,\{x_1,\, \cdots,\, x_n\}\, \subset\, X\, .$$
A quasiparabolic bundle on $X$ with quasiparabolic structure on $S$ is a
holomorphic vector bundle $E$ over $X$ of positive rank together with
a strictly decreasing filtration of subspaces
\begin{equation}\label{e3}
E_{x_i}\,=\, E^1_i\, \supsetneq\, E^2_i \,\supsetneq\, \cdots\, \supsetneq\,
E^{n_i}_i \, \supsetneq\, E^{n_i+1}_i \,=\, 0
\end{equation}
for every $1\, \leq\, i\, \leq\, n$; here $E_{x_i}$ denotes the fiber
of $E$ over the point $x_i\,\in\, X$. A \textit{parabolic structure} on $E$ over $D$ is a
quasiparabolic structure as above together with $n$ increasing sequences of rational numbers
\begin{equation}\label{e5}
0\, \leq\, \alpha^1_i\, <\, \alpha^2_i\, <\,
\cdots\, < \, \alpha^{n_i}_i \, < 1\, , \ \ 1\, \leq\, i\, \leq\, n \, ;
\end{equation}
the rational number $\alpha^j_i$ is called the parabolic weight of the subspace $E^j_i$ in
the quasiparabolic filtration in \eqref{e3}. To clarify, $j$ in $\alpha^j_i$ is an index
and not an exponent. The multiplicity of a parabolic weight
$\alpha^j_i$ at $x_i$ is defined to be the dimension of the complex
vector space $E^j_i/E^{j+1}_i$.
A parabolic vector bundle is a holomorphic vector bundle with a
parabolic structure.

The \textit{parabolic degree} of a parabolic bundle $E_*\,:=\, \left(E,\, \{E^j_i\}, \,\{\alpha^j_i\}\right)$
is defined to be
$$
\text{par-deg}(E_*)\,=\, \text{degree}(E)+\sum_{i=1}^n \sum_{j=1}^{n_i}
\alpha^j_i \dim \left(E^j_i/E^{j+1}_i\right)
$$
\cite[p.~214, Definition~1.11]{MS}, \cite[p.~78]{MY}. The number
$$
\text{par-}\mu(E_*)\,:=\, \frac{\text{par-deg}(E_*)}{{\rm rank}(E)}\, \in\, {\mathbb Q}
$$
is called the {\it parabolic slope}.

Take any holomorphic subbundle $F\, \subset\, E$. For each $x_i\, \in\, S$, the
fiber $F_{x_i}$ has a filtration of subspaces obtained by intersecting the quasiparabolic filtration
of $E_{x_i}$ with the subspace $F_{x_i}\, \subset\, E_{x_i}$. The parabolic weight of a subspace $B\,
\subset\, F_{x_i}$ occurring in this filtration is the maximum of the numbers
$$\{\alpha^j_i\, \,\mid\,\, B\, \subset \, E^j_i\cap F_{x_i}\}\, .$$
This way, the parabolic structure on $E$ produces a parabolic structure on the subbundle $F$.
The resulting parabolic bundle will be denoted by $F_*$.

A parabolic vector bundle $$E_*\,=\, \left(E,\, \{E^j_i\}, \,\{\alpha^j_i\}\right)$$ is called
\textit{stable} (respectively, \textit{semistable}) if for all holomorphic subbundles
$F\, \subsetneq\, E$ of positive rank the following inequality holds:
$$
\text{par-}\mu(F_*) \, <\, \text{par-}\mu(E_*)\ \ \
(\text{respectively, }\, \ \text{par-}\mu(F_*) \, \leq\, \text{par-}\mu(E_*)).
$$

A parabolic bundle is called \textit{polystable} if it is a direct sum of stable
parabolic bundles of same parabolic slope.

Let
\begin{equation}\label{eb0}
\text{End}_P(E_*)\, \subset\, \text{End}(E)\,=\, E\otimes E^*
\end{equation}
be the coherent analytic subsheaf defined by the following condition: A holomorphic
section $s\, \in\, H^0(U,\, \text{End}(E)\big\vert_U)$, where $U\, \subset\, X$ is any
open subset, is a section of $\text{End}_P(E_*)\big\vert_U$ if $s(E^j_i)\, \subset\, E^j_i$
for all $x_i\, \in\, S\bigcap U$ and all $1\, \leq\, j\, \leq\, n_i$. Let
\begin{equation}\label{eb}
\text{End}_N(E_*)\, \subset\, \text{End}_P(E_*)
\end{equation}
be the coherent analytic subsheaf defined by the following condition: A holomorphic
section $s\, \in\, H^0(U,\, \text{End}_P(E_*)\big\vert_U)$, where $U\, \subset\, X$ is any
open subset, is a section of $\text{End}_N(E_*)\big\vert_U$ if $s(E^j_i)\, \subset\, E^{j+1}_i$
for all $x_i\, \in\, S\bigcap U$ and all $1\, \leq\, j\, \leq\, n_i$. It may be mentioned that
the origin of the subscript ``N'' in \eqref{eb} is ``nilpotent''.

\subsection{Connections on a parabolic bundle}\label{se3.2}

Take $X$ and $S$ as in Section \ref{se3.1}. Let $V$ be a holomorphic
vector bundle on $X$. A \textit{logarithmic connection} on $V$ singular over
$S$ is a holomorphic differential operator
$$
D\, :\, V\, \longrightarrow\, V\otimes K_X\otimes {\mathcal O}_X(S)
$$
satisfying the Leibniz identity which says that
\begin{equation}\label{e4}
D(fs)\,=\, fD(s)+ s\otimes df
\end{equation}
for any locally
defined holomorphic function $f$ on $X$ and any locally defined holomorphic section $s$ of $V$. 

Take a point $y\, \in\, S$. The fiber of
$K_X\otimes{\mathcal O}_X(S)$ over $y$ is identified with $\mathbb C$ by the Poincar\'e adjunction
formula \cite[p.~146]{GH}. To explain this isomorphism
\begin{equation}\label{pa}
(K_X\otimes{\mathcal O}_X(S))_y \, \stackrel{\sim}{\longrightarrow}\, {\mathbb C}\, ,
\end{equation}
let $z$ be a holomorphic coordinate function on $X$ defined on an analytic
open neighborhood of $y$ such that $z(y)\,=\, 0$. Then we have an isomorphism
${\mathbb C}\, \longrightarrow\, (K_X\otimes{\mathcal O}_X(S))_y$ that sends any $c\, \in\, \mathbb C$
to $c\cdot \frac{dz}{z}(y)\,\in\, (K_X\otimes{\mathcal O}_X(S))_y$. It is straightforward to check that this map
${\mathbb C}\, \longrightarrow\, (K_X\otimes{\mathcal O}_X(S))_y$ is actually independent of the choice
of the holomorphic coordinate function $z$.

Let $D_V\, :\, V\, \longrightarrow\, V\otimes K_X\otimes{\mathcal O}_X(S)$ be a logarithmic connection
on $V$. From \eqref{e4} it follows that the composition of homomorphisms
$$
V\, \xrightarrow{\,\ D_V\,\ }\, V\otimes K_X\otimes{\mathcal O}_X(S) \, \longrightarrow\,
(V\otimes K_X\otimes{\mathcal O}_X(S))_y\,\stackrel{\sim}{\longrightarrow}\, V_y
$$
is ${\mathcal O}_X$--linear; the above isomorphism $(V\otimes K_X\otimes{\mathcal O}_X(S))_y\,
\stackrel{\sim}{\longrightarrow}\, V_y$ is given by the isomorphism in
\eqref{pa}. Therefore, this composition of homomorphisms produces a
$\mathbb C$--linear homomorphism
$$
{\rm Res}(D_V,\,y)\, :\, V_y\, \longrightarrow\, V_y\, ,
$$
which is called the \textit{residue} of $D_V$ at $y$; see \cite{De}.

Take a parabolic vector bundle $E_*\,=\, \left(E,\, \{E^j_i\}, \,\{\alpha^j_i\}\right)$.

\begin{definition}\label{dlc}
A \textit{connection} on $E_*$ is a logarithmic connection $D$ on $E$, singular over
$S$, such that
\begin{enumerate}
\item $\text{Res}(D,\,x_i)(E^j_i)\, \in\, \text{End}_P(E_*)_{x_i}$
for all $x_i\, \in\, S$ (see \eqref{eb0}),

\item the endomorphism of $E^j_i/E^{j+1}_i$ induced by $\text{Res}(D,\,x_i)$ coincides with
multiplication by the parabolic weight $\alpha^j_i$ for all $1\,\leq\, j\,\leq\, n_i$,
$1\,\leq\, i\, \leq\, n$, and

\item the local monodromy of $D$ around each $x_i\, \in\, S$ is semisimple (an endomorphism
of a finite dimensional complex vector space is called semisimple if it is diagonalizable).
\end{enumerate}
(See \cite[Section~2.2]{BL}.)
\end{definition}

It should be clarified that the third condition in Definition \ref{dlc} is put only because all
connections considered here actually satisfy this condition. For example, an orbifold/equivariant connection automatically
satisfies this condition. In this sense this condition is superfluous.

\begin{lemma}\label{los}
A parabolic vector bundle $E_*$ admits a parabolic connection if and only if $E_*$ is a direct sum of indecomposable
parabolic vector bundles of parabolic degree zero.

Let $E_*$ be a parabolic vector bundle admitting a connection. Then the space of connections on 
$E_*$ is an affine space for the vector space $H^0(X,\, {\rm End}_N(E_*)\otimes 
K_X\otimes{\mathcal O}_X(S))$, where ${\rm End}_N(E_*)$ is defined in \eqref{eb}.
\end{lemma}

\begin{proof}
The first statement is proved in \cite{Bi3}, \cite{BL}.

To prove the second statement, let $D$ be a connection on $E_*$. Take any
$$
\theta\, \in\, H^0(X,\, \text{End}_N(E_*)\otimes K_X\otimes{\mathcal O}_X(S))
$$
and consider the differential operator $D+\theta$; it is a logarithmic connection on $E$. Note that $D+\theta$
satisfies the first two conditions in Definition \ref{dlc}. Therefore, to prove that $D+\theta$ is a connection
on $E_*$ it suffices to show that the local monodromy of $D+\theta$ around each $x_i\,\in\,S$ is semisimple.

Next we recall how the local monodromy around $x_i$ is given by the residue at $x_i$. If $D'$ is a logarithmic
connection whose residue at a point $z$ is $R$, then the local monodromy of $D'$ at $z$ is
$\exp(-2\pi\sqrt{-1}R)$ \cite{De}. Using this property of residue, and the given condition
that $D$ is a connection on $E_*$, it is now straightforward to check
that the local monodromy of $D+\theta$ around each $x_i\, \in\, S$ is semisimple.
\end{proof}

\subsection{Equivariant vector bundles}\label{se3.3}

Take a compact connected Riemann surface $Z$, and let $\Gamma\, \subset\,
\text{Aut}(Z)$ be a finite subgroup of the group of all holomorphic automorphisms of
$Z$. So the quotient $Z/\Gamma$ is a compact connected Riemann surface. We call $Z$ a (ramified)
Galois covering of $Z/\Gamma$; the Galois group is $\Gamma$.

An \textit{equivariant} vector bundle over $Z$ is a holomorphic vector bundle $$F\, 
\longrightarrow\, Z$$ equipped with a lift of the tautological action of $\Gamma$ on $Z$. This 
means that $\Gamma$ acts on $F$ and the action of any $\gamma\, \in\, \Gamma$ on $F$ is a 
holomorphic automorphism of the holomorphic vector bundle $F$ over the automorphism $\gamma$ of 
$Z$.

An equivariant vector bundle $F$ is called equivariantly stable (respectively, 
equivariantly semistable) if for every holomorphic subbundle $F'\, \subset\, F$ such that
\begin{itemize}
\item $0\, <\, {\rm rank}(F')\, <\, {\rm rank}(F)$, and

\item the action of $\Gamma$ on $F$ preserves $F'$,
\end{itemize}
the following inequality holds:
$$
\frac{\text{degree}(F')}{\text{rank}(F')}\, <\,
\frac{\text{degree}(F)}{\text{rank}(F)}\ \
\left(\text{ respectively,}\, \ \frac{\text{degree}(F')}{\text{rank}(F')} \, \leq\,
\frac{\text{degree}(F)}{\text{rank}(F)}\right).
$$

We note that given an equivariant bundle $F$, it is equivariantly semistable if and only if the 
holomorphic vector bundle $F$ is semistable in the usual sense. Indeed, this follows 
immediately from the fact that the Harder--Narasimhan filtration of $F$ is preserved by 
the action of $\Gamma$. If the holomorphic vector bundle $F$ is stable in the usual sense, 
then clearly the equivariant bundle $F$ is equivariantly stable. However, the holomorphic 
vector bundle underlying an equivariantly stable bundle need not be stable in the usual 
sense. To give such an example, let $V_0$ be an irreducible complex $\Gamma$--module
of dimension at least two. Then for the diagonal action of $\Gamma$ on $Z\times V_0$, the
trivial vector bundle $Z\times V_0\, \longrightarrow\, Z$ is equivariantly stable, while this
trivial vector bundle is not stable in the usual sense.

There is an equivalence of categories between the parabolic vector bundles with given 
rational parabolic weights on $X$ and the equivariant bundles on some suitable ramified 
Galois covering of $X$ \cite{Bi1}, \cite{Bo1}, \cite{Bo2}. We recall below some properties 
of this equivalence of categories.

Let $E_*$ be a parabolic bundle such that at any parabolic point $x_i\, \in\, S$
the parabolic weights $\alpha^j_i$ in \eqref{e5} are of the form $\frac{b_{i,j}}{N_{i,j}}$,
where $b_{i,j}$ and $N_{i,j}$ are coprime integers; assume that for every $i$ at least one
of the parabolic weights at $x_i$ is nonzero. Let $N_i$ be the least common multiple of
the integers $N_{i,1},\, \cdots,\, N_{i, n_i}$. If
$$
\varphi\, :\, Z\, \longrightarrow\, X
$$
is a ramified Galois covering, and $F$ is an equivariant vector bundle on $Z$ that corresponds to
$E_*$ by the above mentioned equivalence of categories, then
\begin{itemize}
\item $\varphi$ is ramified over each $x_i\, \in\, S$, and

\item the order of ramification of $\varphi$ at any $y\, \in\, \varphi^{-1}(x_i)$ is an
integral multiple of the above integer $N_i$.
\end{itemize}
The action of the Galois group $\Gamma\,=\, {\rm Gal}(\varphi)$ on the equivariant bundle $F$ produces an
action of $\Gamma$ on the direct image $\varphi_*F$; this action of $\Gamma$ on $\varphi_*F$ is over the
trivial action of $\Gamma$ on $X$. The holomorphic vector bundle $E$ underlying the parabolic bundle
$E_*$ is identified with the invariant part $(\varphi_*F)^\Gamma\,\subset\, \varphi_*F$ for this
action of $\Gamma$.

Conversely, if
\begin{equation}\label{e6}
\varphi\, :\, Z\, \longrightarrow\, X
\end{equation}
is a ramified Galois covering such that
\begin{itemize}
\item $\varphi$ is ramified over each $x_i\, \in\, S$, and

\item the order of ramification of $\varphi$ at any $y\, \in\, \varphi^{-1}(x_i)$ is an
integral multiple of the above integer $N_i$,
\end{itemize}
then there is a unique, up to an equivariant isomorphism, equivariant vector bundle $F$ on $Z$ that
corresponds to the parabolic vector bundle $E_*$.

In view of Corollary \ref{cor1}, given parabolic data, there is a is a ramified Galois covering 
$$
\varphi\, :\, Z\, \longrightarrow\, X
$$
such that
\begin{enumerate}
\item The parabolic bundles of given parabolic type correspond to equivariant bundles on $Z$,

\item the Galois group ${\rm Aut}(Z/X)$ is cyclic, and

\item $\varphi$ is totally ramified over the points of $S$.
\end{enumerate}

A parabolic vector bundle $E_*$ is parabolic stable (respectively, parabolic semistable) if and
only if the corresponding equivariant bundle is equivariantly stable (respectively, equivariantly semistable);
see \cite{Bi1}, \cite{Bo1}, \cite{Bo2} for the details.

As in \eqref{e6}, let $\varphi$ be a ramified Galois covering with Galois group $\Gamma$, where
both $Z$ and $X$ are compact connected Riemann surfaces. Let $F$ be an equivariant bundle on $Z$.

We recall the definition of a holomorphic connection on $F$. A holomorphic connection on a holomorphic vector
bundle $W$ on $Z$ is a holomorphic differential operator
$$
D_W\, :\, W\, \longrightarrow\, W\otimes K_Z,
$$
where $K_Z$ as before is the holomorphic cotangent bundle of $Z$, such that for any locally defined holomorphic
section $s$ of $W$, and any locally defined holomorphic function $f$, the Leibniz identity
$$
D_W(fs) \,=\, f\cdot D_W(s)+ s\otimes df
$$
is satisfied (see \cite{At1}). The Leibniz identity actually ensures that the order of the differential operator $D_W$ is
one.

An \textit{equivariant connection} on the equivariant bundle $F$ on $Z$ is a holomorphic connection $D$ on $F$
such that the action of $\Gamma$ on $F$ preserves $D$.

If $F$ admits a holomorphic connection, then $F$ admits an equivariant connection. Indeed,
if $D_F$ is a holomorphic connection on $F$, then the holomorphic connection
$$
D'_F\,:=\, \frac{1}{\#\Gamma} \sum_{\gamma\in\Gamma}\gamma^* D_F
$$
is preserved by the action of $\Gamma$ on $F$. If $F$ admits an equivariant connection, then
the space of all equivariant connections on $F$ is an affine space for the invariant vector space
\begin{equation}\label{ic2}
H^0(Z, \, \text{End}(F)\otimes K_Z)^\Gamma \, \subset\, H^0(Z, \, \text{End}(F)\otimes K_Z).
\end{equation}

Let $D$ be an equivariant connection on $F$. Then $D$ induces a logarithmic connection on the
direct image $\varphi_* F$; this logarithmic connection on $\varphi_* F$ will be denoted by $\varphi_* D$.
The singular locus of $\varphi_* D$ is precisely the divisor on $X$ over which $\varphi$ is ramified.
The logarithmic connection $\varphi_* D$ on $\varphi_* F$ preserves the subbundle
$(\varphi_* F)^\Gamma\, \subset\, \varphi_* F$ fixed by the
action of the group $\Gamma$. The resulting logarithmic connection on $(\varphi_* F)^\Gamma$
will be denoted by $D'$.

As noted before, if $E_*$ is the parabolic vector bundle on $X$ corresponding to an
equivariant bundle $F$ on $Z$, then the holomorphic vector bundle $E$ underlying $E_*$ is 
$(\varphi_* F)^\Gamma$. Recall the definition connections on a parabolic 
bundle (see Definition \ref{dlc}). The above logarithmic connection $D'$ on $(\varphi_* F)^\Gamma\,=\, E$
given by $\varphi_* D$ is actually 
a connection on the parabolic vector bundle $E_*$ \cite{Bi2}, \cite{Bi3}. In fact this 
produces a natural bijection between the equivariant connections on $F$ and the 
connections on the corresponding parabolic bundle $E_*$ (see \cite{Bi2}, \cite{Bi3}).

Consider $H^0(Z, \, \text{End}(F)\otimes K_Z)^\Gamma$ in \eqref{ic2}. We have
$$
H^0(Z, \, \text{End}(F)\otimes K_Z)^\Gamma\,=\,H^0(X,\, {\rm End}_N(E_*)\otimes
K_X\otimes{\mathcal O}_X(S)),$$
where ${\rm End}_N(E_*)$ is defined in \eqref{eb}. For any
$\theta\, \in\, H^0(Z, \, \text{End}(F)\otimes K_Z)^\Gamma$, let
\begin{equation}\label{tp}
\theta'\, \in\, H^0(X,\, {\rm End}_N(E_*)\otimes
K_X\otimes{\mathcal O}_X(S))
\end{equation}
be the element that corresponds to it by the above isomorphism.

Now from Lemma \ref{los} we have the following:

\begin{lemma}\label{ic}
If a parabolic vector bundle $E_*$ corresponds to an equivariant vector bundle $F$, then
there is a natural bijection between the connections on $E_*$ and the equivariant
connections on $F$. Moreover, for any connection $D_E$ on $E_*$ if $D_F$ is the
corresponding equivariant connection on $F$, then for every
$\theta\, \in\, H^0(Z, \, \text{End}(F)\otimes K_Z)^\Gamma$, the equivariant connection
$D_F+\theta$ corresponds to the connection $D_E+\theta'$ on $E_*$ (see \eqref{tp} for
$\theta'$).
\end{lemma}

\section{Two natural constructions}

\subsection{A canonical connection}

We will now recall a construction from \cite{BB}.
Let $Z$ be a compact connected Riemann surface of genus $g_Z$. Fix a theta characteristic 
${\mathcal L}$ on $Z$. Take a holomorphic vector bundle $V$ on 
$Z$ of rank $r$ and degree zero such that
\begin{equation}\label{ev}
H^0(Z,\, V\otimes{\mathcal L})\,=\, 0\, .
\end{equation}
By the Riemann--Roch theorem,
$$
\dim H^0(Z,\, V\otimes{\mathcal L})- \dim H^1(Z,\, V\otimes{\mathcal L})\,=\, r(g_Z-1)-r(g_Z-1)\,=\, 0\, .
$$
Therefore, we conclude that
\begin{equation}\label{e8}
H^0(Z,\, V\otimes{\mathcal L})\,=\, 0\,=\, H^1(Z,\, V\otimes{\mathcal L})\, .
\end{equation}
Hence, Serre duality gives that
\begin{equation}\label{e8a}
H^0(Z,\, V^*\otimes{\mathcal L})\,=\, 0\,=\, H^1(Z,\, V^*\otimes{\mathcal L})\, .
\end{equation}

\begin{remark}
In any locally complete holomorphic family of vector bundles of rank $r$ and degree zero on $Z$, the condition in \eqref{ev} defines
a divisor on the parameter space. In particular, the condition in \eqref{ev} defined a divisor on the moduli space of semistable
vector bundles on $Z$ of rank $r$ and degree zero; such a divisor is known as a {\it generalized} theta divisor
\cite{BNR}. A generalized theta divisor gives the ample generator of the Picard of the moduli space of semistable vector
bundles on $Z$ of rank $r$ and trivial determinant \cite{DN} (this Picard group is isomorphic to $\mathbb Z$).
\end{remark}

For $i\,=\, 1,\, 2$, let
\begin{equation}\label{e9}
p_i\, :\, Z\times Z\, \longrightarrow\, Z
\end{equation}
be the natural projection to the $i$-th factor. Let
\begin{equation}\label{e10}
\Delta_Z\, \subset\, Z\times Z
\end{equation}
be the reduced diagonal divisor; it will be identified with $Z$ by the map
$Z\, \longrightarrow\, \Delta_Z$ defined by $z\, \longmapsto\, (z,\, z)$.
Note that $$(p^*_1 {\mathcal L})\otimes (p^*_2 {\mathcal L}))\big\vert_{\Delta_Z}
\,=\, K_{\Delta_Z}\,=\, {\mathcal O}_{Z\times Z}(-\Delta_Z)\big\vert_{\Delta_Z}\, ,
$$
where $p_i$ and $\Delta_Z$ are defined in \eqref{e9} and \eqref{e10} respectively; the
above isomorphism $K_{\Delta_Z}\,=\, {\mathcal O}_{Z\times Z}(-\Delta_Z)\big\vert_{\Delta_Z}$
is given by the Poincar\'e adjunction formula \cite[p.~146]{GH}.

Consider the short exact sequence of coherent analytic sheaves on $Z\times Z$
\begin{equation}\label{e11}
0\, \longrightarrow\, (p^*_1(V\otimes{\mathcal L}))\otimes
(p^*_2(V^*\otimes{\mathcal L})) \, \longrightarrow\,
(p^*_1(V\otimes{\mathcal L}))\otimes
(p^*_2(V^*\otimes{\mathcal L}))
\otimes{\mathcal O}_{Z\times Z}(\Delta_Z)
\end{equation}
$$
\longrightarrow\, ((p^*_1 V)\otimes (p^*_2 V^*))\big\vert_{\Delta_Z}\otimes
((p^*_1 {\mathcal L})\otimes (p^*_2 {\mathcal L}){\mathcal O}_{Z\times Z}
(\Delta_Z))\big\vert_{\Delta_Z}\,=\, \text{End}(V) \, \longrightarrow\, 0\, ,
$$
where $\text{End}(V)$ is the torsion sheaf supported on $\Delta_Z\,\subset\, Z\times Z$;
here the above mentioned identification of $Z$ with $\Delta_Z$ is being used. Let
\begin{equation}\label{e12}
H^0(Z\times Z,\, p^*_1(V\otimes{\mathcal L})\otimes
p^*_2(V^*\otimes{\mathcal L})) \, \longrightarrow\,
H^0(Z\times Z,\, p^*_1(V\otimes{\mathcal L})\otimes
p^*_2(V^*\otimes{\mathcal L})
\otimes{\mathcal O}_{Z\times Z}(\Delta_Z))
\end{equation}
$$
\stackrel{\lambda_V}{\longrightarrow}\, H^0(Z,\, \text{End}(V))\,
\longrightarrow\,H^1(Z\times Z,\, p^*_1(V\otimes{\mathcal L})\otimes
p^*_2(V^*\otimes{\mathcal L}))
$$
be the long exact sequence of cohomologies associated to the short exact sequence
of sheaves in \eqref{e11}. Since
$$
H^i(Z\times Z,\, (p^*_1(V\otimes{\mathcal L}))\otimes
(p^*_2(V^*\otimes{\mathcal L})))\,=\, \bigoplus_{k=0}^i H^k(Z,\, V\otimes{\mathcal L})
\otimes H^{i-k}(Z,\, V^*\otimes{\mathcal L})\, ,
$$
using \eqref{e8} and \eqref{e8a} it follows that $H^i(Z\times Z,\,
(p^*_1(V\otimes{\mathcal L}))\otimes (p^*_2(V^*\otimes{\mathcal L})))\,=\, 0$ for all $i$.
Consequently, from \eqref{e12} we get an isomorphism
\begin{equation}\label{e13}
\lambda_V\,:\, H^0(Z\times Z,\, (p^*_1(V\otimes{\mathcal L}))\otimes
(p^*_2(V^*\otimes{\mathcal L}))
\otimes{\mathcal O}_{Z\times Z}(\Delta_Z)) \, \stackrel{\sim}{\longrightarrow}\,
H^0(Z,\, \text{End}(V))\, .
\end{equation}
Let
\begin{equation}\label{e14}
\Psi'_V\,\, :=\,\, (\lambda_V)^{-1}({\rm Id}_V)
\, \in\, H^0(Z\times Z,\, (p^*_1(V\otimes{\mathcal L}))\otimes
(p^*_2(V^*\otimes{\mathcal L}))
\otimes{\mathcal O}_{Z\times Z}(\Delta_Z))
\end{equation}
be the section given by the identity map of $V$ using the isomorphism $\lambda_V$
in \eqref{e13}.

On the other hand, the restriction of $(p^*_1 {\mathcal L})\otimes
(p^*_2 {\mathcal L})\otimes{\mathcal O}_{Z\times Z}(\Delta_Z)$ to the nonreduced divisor
$2\Delta_Z\, \subset\, Z\times Z$ has a canonical trivialization; this means that we have a section
\begin{equation}\label{e15a}
\sigma_Z\, \in\, H^0\left(2\Delta_Z,\, \left((p^*_1 {\mathcal L})\otimes
(p^*_2 {\mathcal L})\otimes{\mathcal O}_{Z\times Z}(\Delta_Z)\right)\big\vert_{2\Delta_Z}\right)
\end{equation}
that identifies $\left((p^*_1 {\mathcal L})\otimes
(p^*_2 {\mathcal L})\otimes{\mathcal O}_{Z\times Z}(\Delta_Z)\right)\big\vert_{2\Delta_Z}$
with ${\mathcal O}_{2\Delta_Z}$
\cite[(3.8)]{BH}, \cite[p.~688, Theorem 2.2]{BR}. This identification between
$\left((p^*_1 {\mathcal L})\otimes (p^*_2 {\mathcal L})\otimes{\mathcal O}_{Z\times Z}(\Delta_Z)\right)\big\vert_{2\Delta_Z}$
and ${\mathcal O}_{2\Delta_Z}$ restricts over $\Delta_Z$ to the natural isomorphism between
$\left((p^*_1 {\mathcal L})\otimes (p^*_2 {\mathcal L})\otimes{\mathcal O}_{Z\times Z}(\Delta_Z)\right)\big\vert_{\Delta_Z}$
and ${\mathcal O}_{\Delta_Z}$ given by the Poincar\'e adjunction formula.
Incorporating the trivialization $\sigma_Z$ in \eqref{e15a} in the
section $\Psi'_V$ in \eqref{e14}, we get a section
\begin{equation}\label{e15}
\Psi_V\,\, \in\,\, H^0\left(2\Delta_Z,\, \left((p^*_1 V)\otimes
(p^*_2V^*)\right)\big\vert_{2\Delta_Z}\right)\, ;
\end{equation}
in other words, we have
\begin{equation}\label{e15b}
(\Psi'_V)\big\vert_{2\Delta_Z}\,=\, (\Psi_V)\big\vert_{2\Delta_Z}
\otimes\sigma_Z\, .
\end{equation}
Since $\lambda_V(\Psi'_V)\,=\, {\rm Id}_V$, the section
$\Psi_V$ in \eqref{e15} satisfies the condition that
$$
\Psi_V\big\vert_{\Delta_Z}\,=\, {\rm Id}_V\, .
$$
Therefore, $\Psi_V$ defines a holomorphic connection on $V$; see \cite{BB}, \cite{BH}
for more details.

Let
\begin{equation}\label{e16}
D_V\, :\, V\, \longrightarrow\, V\otimes K_Z
\end{equation}
be the holomorphic connection on $V$ given by $\Psi_V$ in \eqref{e15}.

Let $\Gamma\, \subset\, {\rm Aut}(Z)$ be a finite cyclic subgroup of the group
of holomorphic automorphisms of $Z$. Assume that the vector bundle $V$ in
\eqref{ev} is equivariant for this subgroup $\Gamma$. Also, assume that the theta
characteristic $\mathcal L$ is $\Gamma$--invariant; note that from Lemma
\ref{lem2} we know that there are $\Gamma$--invariant theta characteristics.

\begin{proposition}\label{prop1}
The action of $\Gamma$ on $V$ preserves the holomorphic connection $D_V$ on
$V$ constructed in \eqref{e16}.
\end{proposition}

\begin{proof}
First note that the tautological 
action of $\Gamma$ on $Z$ has a natural lift to an action of $\Gamma$ on $K_Z$; the action 
of any $\gamma\, \in\, \Gamma$ on $K_Z$ coincides with the automorphism $((d\gamma)^*)^{-1}
\, :\, K_Z\,\longrightarrow\, K_Z$ over the automorphism $\gamma$ of $Z$.

Consider the diagonal action of $\Gamma$ on $Z\times Z$ given by the tautological
action of $\Gamma$ on $Z$. We will prove the following:
\begin{enumerate}
\item The holomorphic line bundle $(p^*_1 {\mathcal L})\otimes (p^*_2 {\mathcal L})$
on $Z\times Z$ is $\Gamma$--equivariant.

\item The $\Gamma$--equivariant line bundle $((p^*_1 {\mathcal L})\otimes (p^*_2 {\mathcal L}))^{\otimes 2}$
is $\Gamma$--equivariantly isomorphic to $(p^*_1 K_Z)\otimes (p^*_2 K_Z)$ (the action of $\Gamma$ on $K_Z$
produces an action of $\Gamma$ on $(p^*_1 K_Z)\otimes (p^*_2 K_Z)$).
\end{enumerate}

Let $N$ be the order of the group $\Gamma$.
Fix a generator $\gamma\, \in\, \Gamma$ of the cyclic group $\Gamma$ of order $N$. Since $\mathcal L$ is
$\Gamma$--invariant, the line bundle $\gamma^*{\mathcal L}$ is holomorphically isomorphic to ${\mathcal L}$.
Fix a holomorphic isomorphism
$$
\beta\, :\, {\mathcal L}\, \stackrel{\sim}{\longrightarrow}\, \gamma^*{\mathcal L}\, .
$$
Now consider the holomorphic isomorphism
$$
\overbrace{\beta\circ \cdots \circ\beta}^{N\text{--times}}\, :\,
{\mathcal L}\, \stackrel{\sim}{\longrightarrow}\, (\gamma^N)^*{\mathcal L}\,=\, {\mathcal L}\, .
$$
It is given by multiplication with some $c_0\, \in\, {\mathbb C}\setminus\{0\}$. Fix a
$c\, \in\, {\mathbb C}\setminus\{0\}$ such that
\begin{itemize}
\item $c^N\,=\, c_0$, and

\item the automorphism $$\frac{1}{c^2}\beta^{\otimes 2}\, :\,
{\mathcal L}^2\,\longrightarrow\, \gamma^*{\mathcal L}^2$$ coincides with the isomorphism
$((d\gamma)^*)^{-1}$ of $K_Z\,=\, {\mathcal L}^2$ with $\gamma^*K_Z\,=\, \gamma^*{\mathcal L}^2$
(note that this condition that
$\frac{1}{c^2}\beta^{\otimes 2}$ coincides with $((d\gamma)^*)^{-1}$ does not depend on the
choice of isomorphism between $K_Z$ and ${\mathcal L}^2$), and

\item the holomorphic automorphism 
$$
\frac{1}{c^2}(p^*_1\beta)\otimes (p^*_1\beta)\, :\, (p^*_1 {\mathcal L})\otimes (p^*_2 {\mathcal L})
\, \stackrel{\sim}{\longrightarrow}\, \gamma^*((p^*_1 {\mathcal L})\otimes (p^*_2 {\mathcal L}))
$$
gives a holomorphic lift of the diagonal action of $\Gamma$ on $Z\times Z$ to a holomorphic
action of $\Gamma$ on the line bundle $(p^*_1 {\mathcal L})\otimes (p^*_2 {\mathcal L})$.
\end{itemize}
It can be shown that there is exactly two solutions for $c$, and they differ only by the sign.

Since the diagonal action of $\Gamma$ on $Z\times Z$ preserves
the divisor $\Delta_Z\, \subset\, Z\times Z$, the holomorphic line bundle
${\mathcal O}_{Z\times Z}(\Delta_Z)$ is equipped with a lift of the action of
$\Gamma$. Also, the action of $\Gamma$ on $V$ evidently produces an action of $\Gamma$ on
$(p^*_1 V)\otimes (p^*_2V^*)$. Using these, together with
the above action of $\Gamma$ on $(p^*_1 {\mathcal L})\otimes (p^*_2 {\mathcal L})$,
it follows that $\Gamma$ acts on the vector bundle $(p^*_1(V\otimes{\mathcal L}))\otimes
(p^*_2(V^*\otimes{\mathcal L}))
\otimes{\mathcal O}_{Z\times Z}(\Delta_Z)$ in \eqref{e11}; this action lifts the
diagonal action of $\Gamma$ on $Z\times Z$. Therefore, $\Gamma$ acts on
$$
H^0(Z\times Z,\, (p^*_1(V\otimes{\mathcal L}))\otimes
(p^*_2(V^*\otimes{\mathcal L}))
\otimes{\mathcal O}_{Z\times Z}(\Delta_Z))\, .
$$

Consider the action of $\Gamma$ on $\text{End}(V)$ given by the action of $\Gamma$
on $V$. The action of $\Gamma$ on $\text{End}(V)$ produces an
action of $\Gamma$ on $H^0(Z,\, \text{End}(V))$. Note that the identity automorphism
$$
{\rm Id}_V\, \in\, H^0(Z,\, \text{End}(V))
$$
is preserved by this action of $\Gamma$ on $H^0(Z,\, \text{End}(V))$. This statement
that ${\rm Id}_V$ is preserved by the action of $\Gamma$ on $H^0(Z,\, \text{End}(V))$
is in fact equivalent to
the statement that the identity map of $V$ is $\Gamma$--equivariant. The isomorphism
$\lambda_V$ in \eqref{e13} is evidently $\Gamma$--equivariant. Since
${\rm Id}_V$ is fixed by the action of $\Gamma$, we now conclude that
$\Psi'_V$ in \eqref{e14} is $\Gamma$--invariant, meaning
\begin{equation}\label{e17}
\Psi'_V\,\,\in \,\, H^0(Z\times Z,\, (p^*_1(V\otimes{\mathcal L}))\otimes
(p^*_2(V^*\otimes{\mathcal L}))
\otimes{\mathcal O}_{Z\times Z}(\Delta_Z))^\Gamma\, .
\end{equation}

Recall that the line bundle $(p^*_1 {\mathcal L})\otimes (p^*_2 {\mathcal L})$ is
$\Gamma$--equivariant, and the divisor $\Delta_Z\, \subset\, Z\times Z$ is preserved
by the action of $\Gamma$.
The section $\sigma_Z$ in \eqref{e15a} is $\Gamma$--invariant; note that the
diagonal action of $\Gamma$ on $Z\times Z$ preserves $2\Delta_Z\, \subset\, Z\times
Z$. Now from \eqref{e17} and \eqref{e15b} we conclude that the section $\Psi_V$
in \eqref{e15} is $\Gamma$--invariant, in other words,
\begin{equation}\label{e18}
\Psi_V\,\, \in\,\, H^0\left(2\Delta_Z,\, \left((p^*_1 V)\otimes
(p^*_2V^*)\right)\big\vert_{2\Delta_Z}\right)^\Gamma\, .
\end{equation}

{}From \eqref{e18} it follows immediately that the holomorphic connection $D_V$ on
$V$ in \eqref{e16} is preserved by the action of $\Gamma$ on $V$.
\end{proof}

\subsection{A splitting}

We will describe a general property of holomorphic connections. Let $N$ be a connected
complex manifold and
\begin{equation}\label{eg}
G\, \subset\, \text{Aut}(N)
\end{equation}
a finite subgroup of the group of holomorphic
automorphisms of $N$. Then the fixed point locus
$$
N^G\,\,:=\,\, \{x\, \in\, N\, \big\vert\,\, G(x)\,=\, x\} \,\, \subset\,\, N ,$$ for the action of $G$ on $N$,
is a submanifold of $N$. Let
$$
F \,\subset\, N^G
$$
be a connected component. Let
\begin{equation}\label{l3}
\iota\,\, :\,\, F\,\, \hookrightarrow\,\, N
\end{equation}
be the inclusion map. So the action of $G$ on $N$ fixes $F$ pointwise.

Take a holomorphic vector bundle $V$ on $N$ equipped with a lift of the action of $G$ on $N$. In other 
words, $V$ is an equivariant holomorphic vector bundle. Consider the Atiyah bundle
$\text{At}(V)$ on $N$ associated to $V$. We recall $\text{At}(V)$ is the sheaf of holomorphic differential 
operators $D$ from $V$ to $V$ of order at most one whose symbol is of the form $s\otimes{\rm 
Id}_V$, where $s$ is a holomorphic vector field on the open subset of $N$ where $D$ is defined 
(see \cite{At1}). So $\text{At}(V)$ fits in a short exact sequence of holomorphic vector
bundles on $N$
\begin{equation}\label{atex}
0\, \longrightarrow\, \text{End}(V) \,=\, V\otimes V^*\, \longrightarrow\, \text{At}(V)\,
\longrightarrow\, TN \, \longrightarrow\, 0\, ,
\end{equation}
where $\text{End}(V)\,=\, V\otimes V^*$ is the sheaf of holomorphic differential
operators from $V$ to $V$ of order zero and the above projection to $TN$ is given by the
symbol map. A holomorphic connection on $V$ is a holomorphic splitting of the short exact
sequence in \eqref{atex} \cite{At1}.

{}From \eqref{atex} we have the following short exact
sequence of holomorphic 
vector bundles on $F$:
\begin{equation}\label{l1}
0\, \longrightarrow\, \text{At}(\iota^* V)\, \stackrel{h_1}{\longrightarrow}\, \text{At}_F(V)
\,:=\, \iota^* \text{At}(V)\,
\stackrel{h_2}{\longrightarrow}\, N_F\, \longrightarrow\, 0\, ,
\end{equation}
where $\iota$ is the map in \eqref{l3} and $N_F$ is the holomorphic normal
bundle of the submanifold $F$.

\begin{lemma}\label{lem3}
Assume that the group $G$ in \eqref{eg} acts on the fibers of the vector bundle
$W\, :=\, \iota^* V$ through pointwise multiplication by a single character of $G$. Then
the short exact sequence in \eqref{l1} admits a canonical holomorphic splitting.
\end{lemma}

\begin{proof}
Since $G$ fixes $F$ pointwise, it acts on the fibers of the normal bundle $N_F$ in \eqref{l1}. All the homomorphisms in \eqref{l1}
are evidently $G$ equivariant. Let
$$
\text{At}_F(V)^{G}\,\subset\,\text{At}_F(V) 
$$
be the fixed point locus for the action of $G$ on $\text{At}_F(V)$. Since $G$ acts on the fibers of $W$ through pointwise
multiplication by a character of $G$, the action of $G$ on $\text{End}(W)\,=\, W\otimes W^*$ is the trivial one.
Recall that $\text{At}(W)$ fits in a short exact sequence
$$
0\, \longrightarrow\, \text{End}(W)\, \longrightarrow\, \text{At}(W)\,
\longrightarrow\, TF\, \longrightarrow\, 0
$$
where all the homomorphisms are $G$--equivariant (see \eqref{atex}). Since $G$ acts
trivially on both $\text{End}(W)$ and $TF$, we
conclude that $G$ acts trivially on $\text{At}(W)$. Consequently, we have
$$
h_1(\text{At}(W))\,\,\subset\,\, \text{At}_F(V)^{G},
$$
where $h_1$ is the homomorphism in \eqref{l1}.

On the other hand, the action of $G$ on any fiber of $N_F$ is nontrivial. From this it follows immediately that
\begin{equation}\label{l4}
h_1(\text{At}(W))\,=\, \text{At}_F(V)^{G}.
\end{equation}
Next, there is a natural holomorphic splitting of the $G$--equivariant bundle
$$
\text{At}_F(V)\,=\, \text{At}_F(V)^{G}\oplus Q\,.
$$
The fiber $Q_x$ over any $x\, \in\, F$ is the sum of all nontrivial irreducible $G$--submodules of the
$G$--module $\text{At}_F(V)_x$. Now from \eqref{l4} we conclude that the natural projection
\begin{equation}\label{l5}
\text{At}_F(V)\,=\, \text{At}_F(V)^{G}\oplus Q\,\longrightarrow\,\text{At}_F(V)^{G}\,=\,
h_1(\text{At}(W))
\end{equation}
gives a holomorphic splitting of \eqref{l1}.
\end{proof}

Consider the short exact sequence of holomorphic vector bundles
\begin{equation}\label{l6}
0\, \longrightarrow\, \iota^* \text{End}(V)\,=\, \text{End}(W)\, \longrightarrow\, \iota^*
\text{At}(V)\,=\, \text{At}_F(V)\, \longrightarrow\, \iota^* TN \, \longrightarrow\, 0
\end{equation}
obtained by restricting \eqref{atex} to $F$.

\begin{corollary}\label{cor2}
Assume that $G$ in \eqref{eg} acts on the fibers of the vector bundle
$W\, :=\, \iota^* V$ through pointwise multiplication by a single character of $G$
(just as in Lemma \ref{lem3}). Then a holomorphic connection on $W$ produces a holomorphic
splitting of the short exact sequence in \eqref{l6}.
\end{corollary}

\begin{proof}
Let $\rho\, :\, \text{At}_F(V)\, \longrightarrow\,\text{At}(W)$ be the homomorphism constructed
in \eqref{l5} that gives the canonical holomorphic splitting of \eqref{l1}. A holomorphic
connection $D$ on $W$ is given by a holomorphic homomorphism $$D'\,:\, TF\, \longrightarrow\,
\text{At}(W)$$ that splits the Atiyah exact sequence for $W$. Consider the holomorphic subbundle
$$
\rho^{-1}(D'(TF))\,\, \subset\,\, \text{At}_F(V)\, .
$$
It is straightforward to check that $\text{At}_F(V)\,=\, \text{End}(W)\oplus \rho^{-1}(D'(TF))$.
Consequently, $\rho^{-1}(D'(TF))$ produces a holomorphic
splitting of the short exact sequence in \eqref{l6}.
\end{proof}

\section{A canonical isomorphism of torsors and group action}

\subsection{An isomorphism of torsors}\label{se5.1}

Take a compact connected Riemann surface $Z$ of genus at least two.
For any integer $r\, \geq\, 1$, let $\mathcal{M}(r)$ denote the moduli space of stable vector
bundles on $Z$ of rank $r$ and degree zero. Let $\mathcal{MC}(r)$ denote the moduli space of 
pairs of the form $(E,\, D)$, where $E\, \in\, \mathcal{M}(r)$ and $D$ is a holomorphic
connection on $E$. Note that any stable vector bundle $V$ of degree zero admits a holomorphic
connection \cite{At1}, \cite{We}, in fact, $V$ admits a unique holomorphic connection whose
monodromy representation is unitary \cite{NS}. Let
\begin{equation}\label{e30}
\Phi\,\, :\,\, \mathcal{MC}(r)\,\, \longrightarrow\,\, \mathcal{M}(r)
\end{equation}
be the forgetful map that sends any $(E,\, D)\,\in\, \mathcal{MC}(r)$ to $E$.
The space of all holomorphic connections
on any holomorphic vector bundle $E\,\in\, \mathcal{M}(r)$ is an affine space for
$H^0(Z,\, \text{End}(E)\otimes K_Z)\,=\, T^*_E \mathcal{M}(r)$. Consequently,
using the projection $\Phi$ in \eqref{e30} the moduli space $\mathcal{MC}(r)$ is a
holomorphic torsor (affine bundle) over $\mathcal{M}(r)$ for the holomorphic
cotangent bundle $T^* \mathcal{M}(r)$.

Fix a theta characteristic ${\mathcal L}$ on $Z$. Consider
\begin{equation}\label{e31}
\widetilde{\Theta}\,\,:=\,\, \{E\, \in\, \mathcal{M}(r)\,\big\vert\,\, H^0(Z,\,
E\otimes{\mathcal L})\,\not=\, 0\}\, \subset\, \mathcal{M}(r)\, ;
\end{equation}
it is an irreducible divisor on $\mathcal{M}(r)$. The holomorphic line bundle
${\mathcal O}_{\mathcal{M}(r)}(\widetilde{\Theta})$ will be denoted by $\Theta$. Let
\begin{equation}\label{e32}
0\, \longrightarrow\,{\mathcal O}_{\mathcal{M}(r)}\, \stackrel{\mu}{\longrightarrow}\,
\text{At}(\Theta)\,\stackrel{\nu}{\longrightarrow}\,T\mathcal{M}(r)\,\longrightarrow\, 0
\end{equation}
be the Atiyah exact sequence for $\Theta$. Consider the projective bundle
\begin{equation}\label{eps}
\psi\, :\, {\mathbb P}(\text{At}(\Theta))\, \longrightarrow\, \mathcal{M}(r)
\end{equation}
that parametrizes all hyperplanes in the fibers of $\text{At}(\Theta)$. Let
\begin{equation}\label{e33}
{\mathcal U}\, \subset\, {\mathbb P}(\text{At}(\Theta))
\end{equation}
be the locus of all $(x,\, H)$, where $x\, \in\, \mathcal{M}(r)$ and $H\, \subset\,
\text{At}(\Theta)_x$ is a hyperplane, such that
\begin{equation}\label{e34}
\text{At}(\Theta)_x\,=\, H\oplus {\rm image}(\mu(x))\,=\, H\oplus{\mathbb C}\, ;
\end{equation}
the map $\mu$ is the one in \eqref{e32}.

Using \eqref{e32} it can be shown that ${\mathcal U}$ in \eqref{e33} is a
holomorphic torsor over $\mathcal{M}(r)$ for the holomorphic
cotangent bundle $T^* \mathcal{M}(r)$. To prove this, take any $x\, \in\,
\mathcal{M}(r)$ and $H\, \in\, {\mathcal U}_x$ (the fiber of ${\mathcal U}$ over $x$).
Then from the decomposition in \eqref{e34} it follows that for any $H'\,\in\, {\mathcal U}_x$,
there is a unique homomorphism
$$
\rho'\,\, :\,\, H \, \longrightarrow\, {\mathbb C}
$$
such that
$$
H'\,=\, \{(v,\, 2r\cdot\rho'(v))\,\,\big\vert\,\, v\, \in\, H\}\, \subset\,
H\oplus {\mathbb C}\,=\, \text{At}(\Theta)_x\, .
$$

On the other hand, the projection $\nu(x)$ in \eqref{e32} identifies $H$ with
$T_x\mathcal{M}(r)$. So $T^*_x\mathcal{M}(r)$ acts on ${\mathcal U}_x$
in \eqref{e33}; the action of
any $\omega\, \in\, T^*_x\mathcal{M}(r)$ sends any $H\, \in\, {\mathcal U}_x$
to
$$
H'\,=\, \{(v,\, 2r\cdot\omega(v))\,\,\big\vert\,\, v\, \in\, H\}\, \subset\,
H\oplus {\mathbb C}\,=\, \text{At}(\Theta)_x
$$
using the decomposition in \eqref{e34}; using the identification of $H$ with 
$T_x\mathcal{M}(r)$ given by $\nu(x)$, we consider $\omega$ as a homomorphism $H \, 
\longrightarrow\, {\mathbb C}$. Consequently, ${\mathcal U}$ in \eqref{e33} is
a holomorphic torsor over $\mathcal{M}(r)$ for $T^* \mathcal{M}(r)$.

Consider the Zariski open subset $\mathcal{M}'(r)\, :=\, \mathcal{M}(r)\setminus \widetilde{\Theta}$,
where $\widetilde{\Theta}$ is the divisor in \eqref{e31}. Let
$$\mathcal{MC}'(r)\, :=\,\Phi^{-1}(\mathcal{M}'(r))
\, \subset\, \mathcal{MC}(r)\ \ \text{ and }\ \ {\mathcal U}'\,:=\, {\mathcal U}\cap
\psi^{-1}(\mathcal{M}'(r))\, \subset\, {\mathcal U}
$$
be the restrictions of the $T^* \mathcal{M}(r)$--torsors $\mathcal{MC}(r)$ and
${\mathcal U}$ respectively (see \eqref{e30} and \eqref{e33}) to this open subset
$\mathcal{M}'(r)$; the map $\psi$ is the projection in \eqref{eps}.
We will show that these two $T^*\mathcal{M}'(r)$--torsors
$\mathcal{MC}'(r)$ and ${\mathcal U}'$ are canonically trivialized.

For any holomorphic vector bundle $V\, \in\, \mathcal{M}'(r)$ consider the holomorphic 
connection $D_V$ on $V$ constructed in \eqref{e16}. So we have a holomorphic section
of the torsor $\mathcal{MC}'(r)$
\begin{equation}\label{s1}
S_C\, :\, \mathcal{M}'(r)\, \longrightarrow\,\mathcal{MC}'(r)
\end{equation}
that sends any $V$ to $D_V$. This gives a holomorphic
trivialization of the $T^*\mathcal{M}'(r)$--torsor $\mathcal{MC}'(r)$. More precisely,
we have an isomorphism of $T^*\mathcal{M}'(r)$--torsors
\begin{equation}\label{m1}
T^*\mathcal{M}'(r)\,\,\stackrel{\sim}{\longrightarrow}\,\, \mathcal{MC}'(r)
\end{equation}
that sends any $\omega\, \in\, T^*_V\mathcal{M}'(r)$, $V\, \in\, \mathcal{M}'(r)$,
to $S_C(V)+\omega$.

To prove that the torsor ${\mathcal U}'$ is trivial, first note that the holomorphic line
bundle $\Theta\,=\, {\mathcal O}_{\mathcal{M}(r)}(\widetilde{\Theta})$ has a canonical
trivialization over $\mathcal{M}'(r)$ given by the constant function $1$. From this
trivialization of $\Theta\big\vert_{\mathcal{M}'(r)}$ we have a decomposition
$$
\text{At}(\Theta)\big\vert_{\mathcal{M}'(r)}\,=\, {\mathcal O}_{\mathcal{M}'(r)}\oplus
T{\mathcal{M}'(r)}\, .
$$
Therefore, the torsor ${\mathcal U}'$ has a holomorphic section
\begin{equation}\label{s2}
\eta\, :\, \mathcal{M}'(r)\,\longrightarrow\,{\mathcal U}'
\end{equation}
that sends any $V\, \in\, \mathcal{M}'(r)$ to the above hyperplane 
$$
T_V{\mathcal{M}'(r)}\, \subset\, \text{At}(\Theta)_V\, .
$$
Consequently, we obtain a holomorphic
trivialization of the $T^*\mathcal{M}'(r)$--torsor ${\mathcal U}'$. In other words,
we have an isomorphism of $T^*\mathcal{M}'(r)$--torsors
\begin{equation}\label{m2}
T^*\mathcal{M}'(r)\,\,\stackrel{\sim}{\longrightarrow}\,\, {\mathcal U}'
\end{equation}
that sends any $\omega\, \in\, T^*_V\mathcal{M}'(r)$, $V\, \in\, \mathcal{M}'(r)$,
to $\eta(V)+ \omega$, where $\eta$ is the section in \eqref{s2}.

Combining \eqref{m1} and \eqref{m2} we obtain an isomorphism of $T^*\mathcal{M}'(r)$--torsors
\begin{equation}\label{m3}
\mathcal{MC}'(r)\,\,\stackrel{\sim}{\longrightarrow}\,\, {\mathcal U}'
\end{equation}

The following result was proved in \cite{BH}.

\begin{theorem}[{\cite[p.~4, Corolary 1.2]{BH}}]\label{thm1}
The isomorphism of $T^*\mathcal{M}'(r)$--torsors in \eqref{m3} extends to a
holomorphic isomorphism of $T^*\mathcal{M}(r)$--torsors
$$
\Psi\,\, :\,\, \mathcal{MC}(r)\,\,\stackrel{\sim}{\longrightarrow}\,\, {\mathcal U}
$$
over entire $\mathcal{M}(r)$.
\end{theorem}

A stable vector bundle of rank $r$ degree zero on $Z$ has a unique holomorphic connection whose
monodromy representation $\pi_1(Z,\, z_0) \, \longrightarrow\, \text{GL}(r,{\mathbb C})$ has
the property that its image is contained in ${\rm U}(r)$ \cite{NR}. Therefore, we get a
$C^\infty$ section
$$
\tau_U\, :\, \mathcal{M}(r)\,\longrightarrow \mathcal{MC}(r)
$$
of the torsor in \eqref{e30} by sending any stable vector bundle $V\, \in\, \mathcal{M}(r)$
to $(V,\, D_U)$, where $D_U$ is the unique holomorphic connection on $V$ with unitary
monodromy. It should be clarified that the map $\tau_U$ is \textit{not} holomorphic.

The holomorphic line bundle $\Theta$ on $\mathcal{M}(r)$ in \eqref{e32} has a natural Hermitian 
structure \cite{Qu}, which is known as the Quillen metric. The moduli space
$\mathcal{M}(r)$ has a natural K\"ahler structure \cite{AB}. The curvature of the Chern connection
on $\Theta$ corresponding to the Quillen metric is given by the K\"ahler form
on $\mathcal{M}(r)$ \cite{Qu}. This Chern connection
on $\Theta$ corresponding to the Quillen metric gives a $C^\infty$ section
$$
\tau_Q\, :\, \mathcal{M}(r)\,\longrightarrow \, \mathcal{U}
$$
of the $T^*\mathcal{M}(r)$--torsor $\mathcal{U}$ in \eqref{e33}. It
should be clarified that $\tau_Q$ is \textit{not} holomorphic.

The following theorem was proved in \cite{BH}.

\begin{theorem}[{\cite[p.~4, Theorem 1.1 and Corolary 1.2]{BH}}]\label{thm1p}
The isomorphism $\Psi$ is Theorem \ref{thm1} takes the section $\tau_U$ to the
section $\tau_Q$, in other words,
$$\Psi\circ\tau_U\,=\, \tau_Q\, .$$
\end{theorem}

\subsection{Action of automorphisms of Riemann surface}\label{se5.2}

Continuing with the notation of Section \ref{se5.1}, fix
a cyclic subgroup $\Gamma\, \subset\, \text{Aut}(Z)$. Note that $\Gamma$ is a finite subgroup because
the genus of $Z$ is at least two.

The action of $\Gamma$ on $Z$ produces a holomorphic action of $\Gamma$ on $\mathcal{M}(r)$; 
the action of any $\gamma\,\in\,\Gamma$ on $\mathcal{M}(r)$ sends any vector bundle $V\,\in\,
\mathcal{M}(r)$ to $(\gamma^{-1})^* V$. Similarly, $\Gamma$ acts on $\mathcal{MC}(r)$; 
the action of any $\gamma\,\in\,\Gamma$ on $\mathcal{MC}(r)$ sends any $(V,\, D_V)$ to
$((\gamma^{-1})^* V,\, (\gamma^{-1})^* D_V)$. The map $\Phi$ in \eqref{e30} is evidently
$\Gamma$--equivariant.

Now choose the theta characteristic $\mathcal L$ in Section \ref{se5.1} to be
$\Gamma$--invariant; from Lemma \ref{lem2} we know that such a theta characteristic exists.
For a holomorphic vector bundle $V$ on $Z$ and any $\gamma\, \in\, \Gamma$, we have
$$
H^i(Z,\, (\gamma^* V)\otimes {\mathcal L})\,=\, H^i(Z,\, (\gamma^* V)\otimes (\gamma^*
{\mathcal L}))\,=\,H^i(Z,\, \gamma^* (V\otimes {\mathcal L}))\,=\, H^i(Z,\,V\otimes{\mathcal L}).
$$
So, $V\, \in\, \widetilde{\Theta}$ (see \eqref{e31}) if and only if
$H^i(Z,\,V\otimes{\mathcal L})\,=\, 0$ if and only if $H^i(Z,\, (\gamma^* V)\otimes {\mathcal L})
\,=\, 0$ if and only if $\gamma^*V\, \in\, \widetilde{\Theta}$. In other words,
the divisor $\widetilde{\Theta}$ is preserved by the action of 
$\Gamma$ on $\mathcal{M}(r)$. Therefore, the holomorphic line bundle $\Theta\,=\, {\mathcal 
O}_{\mathcal{M}(r)}(\widetilde{\Theta})$ is $\Gamma$--equivariant. The actions of $\Gamma$ on 
$\mathcal{M}(r)$ and $\Theta$ together produce an action of $\Gamma$ on $\text{At}(\Theta)$.
This in turn produces an action of $\Gamma$ on $\mathcal U$ (constructed in \eqref{e33}).
The natural projection ${\mathcal U}\, \longrightarrow\, \mathcal{M}(r)$ is evidently
$\Gamma$--equivariant.

The section $S_C$ in \eqref{s1} is $\Gamma$--equivariant. Indeed, this follows immediately from 
the fact that the theta characteristic $\mathcal L$ is $\Gamma$--invariant. Consequently, the 
isomorphism in \eqref{m1} is $\Gamma$--equivariant; the action of $\Gamma$ on 
$T^*\mathcal{M}'(r)$ is given by the action of $\Gamma$ on $\mathcal{M}'(r)$.

Since the divisor $\widetilde{\Theta}$ in \eqref{e31} is preserved by the action of $\Gamma$ on 
$\mathcal{M}(r)$, the action of $\Gamma$ on $\mathcal U$ preserves ${\mathcal U}'\, \subset\, 
{\mathcal U}$. Moreover, the section $\eta$ in \eqref{s2} is $\Gamma$--equivariant. Therefore, 
the isomorphism in \eqref{m2} is $\Gamma$--equivariant.

Since the isomorphisms in \eqref{m1} and \eqref{m2} are both $\Gamma$--equivariant, the
isomorphism in \eqref{m3} is also $\Gamma$--equivariant. We put this
down in the form of the following lemma.

\begin{lemma}\label{lem4}
The isomorphism $\Psi$ in Theorem \ref{thm1} is $\Gamma$--equivariant.
\end{lemma}

\section{Isomorphism of torsors on moduli of parabolic bundles}

\subsection{Two torsors on moduli of parabolic bundles}

Fix $(X,\, S)$ as in Section \ref{se3.1}, and fix an integer $r\, \geq\, 1$. Fix parabolic data 
for rank $r$ bundles; this means choosing parabolic weights with multiplicities for
each point $x_i\, \in\, S$ such that the total number of parabolic weights at $x_i$ is $r$.
We assume that the total sum of parabolic weights over $S$ is an integer.

Let ${\mathcal N}_P(r)$ denote the moduli space of stable parabolic bundles of rank $r$ and 
parabolic degree zero, with the given parabolic structure \cite{MS}, \cite{MY}. Note that the 
above condition that the total sum of parabolic weights over $S$ is an integer is actually 
needed to ensure that ${\mathcal N}_P(r)$ is nonempty. We recall that ${\mathcal N}_P(r)$ is a 
smooth irreducible quasiprojective variety. If the parabolic data is such that any semistable 
parabolic bundle is stable, then ${\mathcal N}_P(r)$ is in fact a projective variety. For
any parabolic bundle $E_*\, \in\, {\mathcal N}_P(r)$, we have
\begin{equation}\label{e35}
T_{E_*}{\mathcal N}_P(r)\,=\, H^1(X,\, \text{End}_P(E_*))\ \ \, \text{ and }\ \ \,
T^*_{E_*}{\mathcal N}_P(r)\,=\, H^0(X,\, \text{End}_N(E_*)\otimes K_X)
\end{equation}
\cite{MS}, \cite{Yo}; see \eqref{eb} and \eqref{eb0} for $\text{End}_N(E_*)$ and $\text{End}_P(E_*)$
respectively.

We will construct two holomorphic $T^*{\mathcal N}_P(r)$--torsors over ${\mathcal N}_P(r)$.

Let ${\mathcal NC}_P(r)$ denote the moduli space of pairs of the form $(E_*,\,D)$, where $E_*\, \in\, {\mathcal N}_P(r)$
is a stable parabolic bundle and $D$ is a connection on $E_*$ (see Definition \ref{dlc}). Let
\begin{equation}\label{e36}
P\, :\, {\mathcal NC}_P(r)\, \longrightarrow\, {\mathcal N}_P(r)
\end{equation}
be the forgetful map that sends any $(E_*,\, D)$ to $E_*$. From Lemma \ref{los} and \eqref{e35} it follows that
$P$ in \eqref{e36} makes ${\mathcal NC}_P(r)$ a holomorphic $T^*{\mathcal N}_P(r)$--torsor over ${\mathcal N}_P(r)$.

On ${\mathcal N}_P(r)$ there are natural ample line bundles constructed using the determinant of 
cohomology combined with the quasiparabolic filtrations over the parabolic points and the 
parabolic weights; see \cite{BRag}, \cite{NR}, \cite{LS}, \cite{Su} for the details. Since 
their construction is not very relevant here, we refrain from recalling it. These line bundles 
can be described using the identification between the parabolic bundles and the equivariant 
bundles mentioned in Section \ref{se3.3}; we will briefly recall this description of the line 
bundles.

As briefly described in Section \ref{se3.3}, given parabolic data, there is a is a ramified 
Galois covering
\begin{equation}\label{vz}
\varphi\, :\, Z\, \longrightarrow\, X
\end{equation}
such that the parabolic bundles of given parabolic type correspond to equivariant bundles on 
$Z$. It was also noted in Section \ref{se3.3} that the pair $(Z,\, \varphi)$ can be so chosen 
that $\Gamma\, :=\, \text{Aut}(Z/X)$ is cyclic. The moduli space ${\mathcal N}_P(r)$ is 
identified with a component of the fixed point locus $\mathcal{M}(r)^\Gamma$ (see \eqref{e30}). 
Then the restriction of the line bundle $\Theta$ (see \eqref{e32} and \eqref{e31}) to
\begin{equation}\label{es}
{\mathcal N}_P(r)\, \subset\, \mathcal{M}(r)^\Gamma\, \subset\, \mathcal{M}(r)
\end{equation}
is a power of the determinant bundle (see \cite{BRag}). This restriction of $\Theta$ to
${\mathcal N}_P(r)$ will be denoted by $\Theta_P$.

Consider the Atiyah exact sequence for $\Theta_P$
\begin{equation}\label{e37}
0\, \longrightarrow\,{\mathcal O}_{{\mathcal N}_P(r)}\, \stackrel{p}{\longrightarrow}\,
\text{At}(\Theta_P)\,\stackrel{q}{\longrightarrow}\,T{\mathcal N}_P(r)\,\longrightarrow\, 0.
\end{equation}
and let
$$
\xi\, :\, {\mathbb P}(\text{At}(\Theta_P))\, \longrightarrow\,{\mathcal N}_P(r)
$$
be the projective bundle that parametrizes all hyperplanes in the fibers of
$\text{At}(\Theta_P)$. Set
\begin{equation}\label{e38}
{\mathcal U}_P\, \subset\, {\mathbb P}(\text{At}(\Theta_P))
\end{equation}
to be the locus of all $(x,\, H)$, where $x\, \in\, {\mathcal N}_P(r)$ and $H\, \subset\,
\text{At}(\Theta_P)_x$ is a hyperplane, such that
\begin{equation}\label{e39}
\text{At}(\Theta_P)_x\,=\, H\oplus {\rm image}(p(x))\,=\, H\oplus{\mathbb C}\, ;
\end{equation}
the map $p$ is the one in \eqref{e37}. Let
\begin{equation}\label{e40}
Q\, :\, {\mathcal U}_P\, \longrightarrow\, {\mathcal N}_P(r)
\end{equation}
be the natural projection. Take any $x\, \in\,{\mathcal N}_P(r)$ and $H\, \in\, Q^{-1}(x)$.
Then from the decomposition in \eqref{e39} it follows
that for any $H'\, \in\, Q^{-1}(x)$, there is a unique homomorphism
$$
\rho'\,\, :\,\, H \, \longrightarrow\, {\mathbb C}
$$
such that
$$
H'\,=\, \{(v,\, 2r\cdot\rho'(v))\,\,\big\vert\,\, v\, \in\, H\}\, \subset\,
H\oplus {\mathbb C}\,=\, \text{At}(\Theta_P)_x\, .
$$
This shows that ${\mathcal U}_P$ in \eqref{e38} is a holomorphic $T^*{\mathcal N}_P(r)$--torsor 
over ${\mathcal N}_P(r)$. For any $x\, \in\, {\mathcal N}_P(r)$, the action of $\omega\, \in\, 
T^*_x{\mathcal N}_P(r)$ on $Q^{-1}(x)$ sends any hyperplane $H\, \in\, Q^{-1}(x)$ to
$$
H'\,=\, \{(v,\, 2r\cdot\omega(v))\,\,\big\vert\,\, v\, \in\, H\}\, \subset\,
H\oplus {\mathbb C}\,=\, \text{At}(\Theta_P)_x
$$
using the decomposition in \eqref{e39}; using the identification of $H$ with
$T_x\mathcal{M}(r)$ given by $q(x)$ (see \eqref{e37}) we consider $\omega$
as a homomorphism $H \,\longrightarrow\, {\mathbb C}$. 

\subsection{Holomorphic isomorphism of torsors}

\begin{theorem}\label{thm2}
There is a natural holomorphic isomorphism
$$\Psi_P\, :\, {\mathcal NC}_P(r)\, \longrightarrow\, {\mathcal U}_P$$
between the two holomorphic $T^*{\mathcal N}_P(r)$--torsors ${\mathcal NC}_P(r)$ and ${\mathcal
U}_P$ (see \eqref{e36} and \eqref{e40}).
\end{theorem}

\begin{proof}
As mentioned before, ${\mathcal N}_P(r)$ is a connected component of $\mathcal{M}(r)^\Gamma$
(see \eqref{es}). Consider the projection $\Phi$ in \eqref{e30}. Since it is
$\Gamma$--equivariant, the action of $\Gamma$ on $\mathcal{MC}(r)$
preserves $\Phi^{-1}({\mathcal N}_P(r))$. It is now straightforward to check
that 
\begin{equation}\label{ec4}
{\mathcal NC}_P(r)\,=\,\Phi^{-1}({\mathcal N}_P(r))^\Gamma\, \subset\, \Phi^{-1}({\mathcal N}_P(r)),
\end{equation}
where ${\mathcal NC}_P(r)$ is the moduli space in \eqref{e36}. Indeed, this is a consequence
of Lemma \ref{ic}.

In Section \ref{se5.2} we saw that the actions of $\Gamma$ on 
$\mathcal{M}(r)$ and $\Theta$ together produce an action of $\Gamma$ on
$\mathcal U$ (constructed in \eqref{e33}).
Since the natural projection
\begin{equation}\label{ec3}
{\mathcal U}\, \longrightarrow\, \mathcal{M}(r)
\end{equation}
is
$\Gamma$--equivariant, the action of $\Gamma$ on ${\mathcal U}$ preserves the
inverse image of ${\mathcal N}_P(r)\, \subset\, \mathcal{M}(r)$ for the projection
in \eqref{ec3}; this inverse image will be denoted by $\widehat{\mathcal U}$.
Clearly, we have
\begin{equation}\label{ec5}
{\mathcal U}_P \,=\, \widehat{\mathcal U}^\Gamma\,\subset\, \widehat{\mathcal U}\, .
\end{equation}

Now the theorem follows from Lemma \ref{lem4} and the isomorphisms in \eqref{ec4} and
\eqref{ec5}.
\end{proof}

Any stable parabolic bundle $E_*\, \in\, {\mathcal N}_P(r)$ of parabolic degree zero has a 
unique connection whose monodromy homomorphism $\pi_1(X\setminus S,\, x_0)\, \longrightarrow\, 
\text{GL}(r, {\mathbb C})$ has its image contained in $\text{U}(r)$ \cite{MS}. Therefore,
we obtain a $C^\infty$ section of the projection $P$ in \eqref{e36}
\begin{equation}\label{ta}
\tau_{UP}\,\,:\, \, {\mathcal N}_P(r)\, \longrightarrow\, {\mathcal NC}_P(r)
\end{equation}
that sends any $E_*\, \in\, {\mathcal N}_P(r)$ to $(E_*,\, D_E)$, where $D_E$ is the unique
connection on $E_*$ whose monodromy homomorphism has its image contained in $\text{U}(r)$.

Consider the Quillen connection on the holomorphic line bundle $\Theta_P$ \cite{Qu},
\cite{BRag}. It produces a $C^\infty$ section
$$
\tau_{QP}\,\,:\, \, {\mathcal N}_P(r)\, \longrightarrow\, {\mathcal U}_P
$$
of the holomorphic $T^*{\mathcal N}_P(r)$--torsor ${\mathcal U}_P$ in \eqref{e38}.

Theorem \ref{thm1p} has the following immediate consequence:

\begin{corollary}\label{cor3}
The isomorphism $\Psi_P$ in Theorem \ref{thm2} takes the section $\tau_{UP}$ to
the section $\tau_{QP}$, or in other words, $$\Psi_P\circ\tau_{UP}\,=\, \tau_{QP}\, .$$
\end{corollary}

\begin{remark}
The isomorphism $\Psi_P$ in Theorem \ref{thm2} is actually independent of the choice of the covering $\varphi$ in
\eqref{vz}. In fact, the two sections the sections $\tau_{UP}$ and $\tau_{QP}$ in Corollary \ref{cor3} are independent of
the choice of $\varphi$. Therefore, from Corollary \ref{cor3} it follows that the isomorphism $\Psi_P$ in Theorem \ref{thm2} is
independent of the choice of the covering $\varphi$.
\end{remark}

\begin{remark}
As mentioned in the introduction, neither of the two sections $\tau_{UP}$ and $\tau_{QP}$ can be constructed staying within the
realm of complex analytic geometry. But Corollary \ref{cor3} says that one determines the other through algebraic geometry. In this
sense, these two sections are deeply interconnected. See \cite{BFPT} for similar relationship between two different
$C^\infty$ sections. In \cite{BFPT}, one section is given by the Quillen connection on the Hodge bundle over the moduli space
${\mathcal M}_g$ of smooth complex projective curves of genus $g\, \geq\, 2$. The other section is given by the uniformization of
Riemann surfaces.
\end{remark}

\begin{remark}
Let us consider the set-up of Theorem \ref{thm2} without the assumption that all the parabolic weights are rational numbers.
The Mehta--Seshadri correspondence between parabolic bundles and unitary representations, which was proved in \cite{MS}
under the assumption that the parabolic weights are rational numbers, is known to extend to the situation where the
parabolic weights are real numbers. This was proved by Biquard \cite{Biq}; a vastly general result was proved in \cite{Si}.
Consequently, the $T^*{\mathcal N}_P(r)$--torsor ${\mathcal NC}_P(r)$ over ${\mathcal N}_P(r)$, along with its
$C^\infty$ section $\tau_{UP}$ (see \eqref{ta}), are there when the parabolic weights are real numbers. 
It is not difficult to see that the $T^*{\mathcal N}_P(r)$--torsor ${\mathcal U}_P$ also exists over ${\mathcal N}_P(r)$.
To explain this, if $L$ is a holomorphic line bundle on a complex manifold $Z$, then $L^{\otimes \lambda}$ makes sense only
when $\lambda$ is an integer. But sheaf of connections of $L^{\otimes\lambda}$ makes sense even if $\lambda$ is a complex number.
Indeed, if
$$
0\, \longrightarrow\, {\mathcal O}_Z \, \longrightarrow\, \text{At}(L) \, \longrightarrow\, TZ \, \longrightarrow\, 0
$$
is the Atiyah exact sequence for $L$, then the Atiyah exact sequence for $L^{\otimes \lambda}$ is the exact sequence obtained
from the above sequence using the endomorphism of ${\mathcal O}_M$ given by the multiplication with $\lambda$. Therefore, it
makes sense to ask whether Theorem \ref{thm2} remains valid in the more general set-up where the parabolic weights are allowed to 
be irrational numbers. The construction of the isomorphism $\Psi_P$ in Theorem \ref{thm2} does not shed any light on how it
changes when the parabolic weights change. For this reason any attempt to answer the question lies outside the scope of
this paper.
\end{remark}

\section*{Acknowledgements}

The author thanks the referee for helpful comments. This work was partly supported by a
J. C. Bose Fellowship.


\end{document}